\documentclass[onefignum,onetabnum]{siamart171218}

\usepackage{lipsum}
\usepackage{amsfonts}
\usepackage{graphicx}
\usepackage{epstopdf}
\usepackage{xcolor}

\usepackage{tikz}
\usetikzlibrary{arrows}

\usepackage{amssymb,amsfonts,latexsym,amstext,epsfig}

\usepackage{algorithmic}
\ifpdf
  \DeclareGraphicsExtensions{.eps,.pdf,.png,.jpg}
\else
  \DeclareGraphicsExtensions{.eps}
\fi

\newsiamremark{remark}{Remark}
\newsiamremark{hypothesis}{Hypothesis}
\crefname{hypothesis}{Hypothesis}{Hypotheses}
\newsiamthm{claim}{Claim}
\newsiamthm{cor}{Corollary}
\crefname{corollary}{Corollary}{Corollary}

\headers{An optimization parameter for seriation of noisy data}{M. Ghandehari and J. Janssen}

\title{An optimization parameter for seriation of noisy data\thanks{Submitted to the editors DATE.
\funding{The authors gratefully acknowledge supports from UDRF and NSERC.}}}

\author{Mahya Ghandehari\thanks{Department of Mathematical Sciences, University of Delaware, Newark, DE
  (\email{mahya@udel.edu}).}
\and Jeannette Janssen\thanks{Department of Mathematics and Statistics, Dalhousie University, Halifax, Canada
  (\email{Jeannette.Janssen@dal.ca}).}}

\usepackage{amsopn}

\begin{document}

\newcommand{\Nnn}{{\mathbb N}}
\newcommand{\Ppp}{{\mathbb P}}
\newcommand{\Eee}{{\mathbb E}}
\newcommand{\Zzz}{{\mathbb Z}}
\newcommand{\Qqq}{{\mathbb Q}}
\newcommand{\Rrr}{{\mathbb R}}
\newcommand{\va}{{\mathbf a}}
\newcommand{\vv}{{\mathbf v}}
\newcommand{\xx}{{\mathbf x}}
\newcommand{\uu}{{\mathbf u}}
\newcommand{\UR}{{\rm UR}}
\newcommand{\LL}{{\rm LL}}
\newcommand{\ext}{{\rm ext}}
\newcommand{\cB}{{\mathcal B}}
\newcommand{\cC}{{\mathcal C}}
\newcommand{\cD}{{\mathcal D}}
\newcommand{\cA}{{\mathcal A}}
\newcommand{\cW}{{\mathcal W}}
\newcommand{\gm}{\Gamma_{1}}
\newcommand{\An}{{\mathcal A}_n}
\newcommand{\jj}[1]{{\bf [~Jeannette:\ } {\em #1}{\bf~]}}
\newcommand{\blue}[1]{#1}
\newcommand{\pink}[1]{#1}

\maketitle

\begin{abstract}
 \blue{A square symmetric} matrix is a Robinson similarity matrix if entries in its rows and columns are non-decreasing when moving towards the diagonal. A Robinson similarity matrix can be viewed as the affinity matrix between objects arranged in linear order, where objects closer together have higher affinity. We define a new parameter, $\gm$, which measures how badly a given matrix fails to be Robinson similarity. Namely, a matrix is Robinson similarity precisely when its $\gm$ attains zero, and a matrix with small $\gm$ is close (in the normalized $\ell^1$-norm) to a Robinson similarity matrix. Moreover, both $\gm$  and the Robinson similarity approximation can be computed in polynomial time. Thus, our parameter recognizes Robinson similarity matrices which are perturbed by noise, and can therefore be a useful tool in the problem of seriation of noisy data. 
\end{abstract}



\section{Introduction}
Many real-life networks, such as online social networks, biological networks and neural networks, are manifestations of an underlying (hidden) spatial reality. For example, members of a social network can be identified with nodes placed in a metric space, whose coordinates represent the interests, backgrounds, and other significant features of the users. The formation of the network is then modeled as
a stochastic process, where the probability of a link occurring between two nodes decreases as their
metric distance increases. A fundamental and challenging problem in the analysis of a social network (or any
other spatial network) is to uncover its ``hidden spatial layout'', {\it i.e.}~to identify the metric space
representation of the network. Analysis and visualization of data becomes considerably more tractable when the dataset is presented according to its spatial reality.

The classical seriation problem, introduced by Robinson in \cite{robinson}, can be viewed as the special case of the spatial layout problem, restricted to one dimension.
The objective of the seriation problem is to order a set of items so that similar items are placed closer to each other. The seriation question translates in a natural way into a question regarding symmetric matrices.
A symmetric matrix is a \emph{Robinson similarity } matrix, or \emph{Robinson matrix} for short, if its entries are non-decreasing when moving towards the main diagonal in each row or column. A symmetric matrix $A$ is said to be  \emph{Robinsonian}, if it becomes a Robinson matrix after simultaneous application of a permutation $\pi$ to its rows and columns. In that case, the permutation $\pi$ is called a \emph{Robinson ordering}  of $A$. If the entries of the symmetric matrix $A=[A_{i,j}]$ represent  similarity of items $i$ and $j$,  then the  Robinson ordering represents a linear arrangement of the items so that similar items are placed closer together.

The problem of recognizing  Robinsonian  matrices, and finding their Robinson orderings, can be solved in polynomial time.  See \cite{mirkin-rodin} for the first polynomial time algorithm for this problem, and \cite{seston,prea-fortin,laurent2016,laurent2017} for more recent efficient algorithms. Most of these algorithms are based on a similar principle; namely the connection between Robinsonian similarity matrices and unit interval graphs (\cite{laurent2016,laurent2017}) or interval (hyper) graphs (\cite{mirkin-rodin,seston,prea-fortin}). A spectral algorithm based on reordering the matrix according to the components of the second eigenvector of the Laplacian, or the Fiedler vector, was given in \cite{atkins}, and was then applied to the \emph{ranking problem} in \cite{fogel}.

The seriation problem has diverse and significant applications, from its origin in archeological studies to recent applications to ecology and sociology. For a historical overview of the problem and its diverse applications, see \cite{Liiv2010}. 
In most of these applications, it is natural to expect the data to be noisy. In that case, the optimal reordering of a data-derived matrix may not be itself a Robinson matrix, but it will be close to one. The question then becomes, to which extent a given matrix resembles a Robinson matrix. 
All of  the algorithms mentioned in the previous paragraph only apply  to noise-free Robinsonian similarity matrices. 

In the presence of noise, the goal of the seriation problem is to find an ``almost Robinson'' ordering of a given matrix, {\it i.e.} an ordering for which the reordered matrix is closest to being Robinson. This question turned out to be much more challenging than the error-free analogue. In fact,  it is shown in \cite{chepoi-fichet} that the problem of finding a reordering and a Robinson matrix which is the best $\ell^\infty$-approximation  is NP-hard. In \cite{Chepoi2011}  a factor 16  approximation algorithm is given for the case of $\ell^\infty$. NP-hardness for a number of related problems is established in \cite{Barthelemy-Brucker}, where approximation by $\ell^p$ distance is considered. Specifically, it is shown that for an integer $p$, the problem of finding {\sl proper strong Robinson relations} within specified $\ell^p$ distances of a given matrix is NP-complete. A proper strong Robinson relation corresponds to an appropriate relabelling of the original matrix together with a Robinson matrix with certain additional, stronger properties. 
In \cite{Rigollet} (together with the references therein) a statistical approach to the problem of seriation with noise is developed, where the error of  Robinson approximation is measured by the Frobenius norm.

If the appropriate labelling is given, then the problem becomes more tractable. Now the problem is that of finding a Robinson matrix closest to a given matrix. For the $\ell^\infty$-norm, it is known that this problem can be solved in polynomial time since an explicit closed form for the
optimal solution can be easily given (cf.\ \cite{Seminaroti2016}). The problem of finding the best $\ell^1$-approximation can be formulated as a linear program: Minimize the linear function $\| A-R\|_1$ subject to the constraint that $R$ is Robinson similarity. The constraint can be expressed with $O (n^3)$ inequalities, and thus the problem can be solved in polynomial time.

In this article, we develop new methods and algorithms which can be used for seriation of noisy data. Our focus here is on formalizing the notion of a matrix being ``almost Robinson.'' To do so,  we introduce a parameter, which we call $\gm$, that measures how much the local structure of a matrix resembles being a Robinson similarity. Namely, $\gm$  sums the magnitude  of local violations to the Robinson similarity property, and achieves the value 0 precisely when the matrix is Robinson.  {This parameter is a natural tool for the seriation problem,} and has been used in practice as a heuristic to measure the amount
of deviation from a Robinson form (see \cite{Chen2002,Hahsler2008}).
Moreover,  $\gm$ is simple to formulate, and can be computed in linear time. The main goal of this paper is to show that if the number and magnitude of local violations is small, then the matrix is indeed close, in the sense of  $\ell^1$-norm, to  a Robinson matrix.  Precisely, we prove in Theorem \ref{thm:main} that for every given matrix $A$,  there exists a Robinson matrix $R$ so that $\|A-R\|_1\leq \blue{26} \gm(A)^{1/3}$.
In addition, we give a polynomial time algorithm to compute the Robinson approximation $R$ which fulfills the above inequality. 

A proposed application of this work is a novel scheme for treating the seriation problem of noisy data, in which the selection of the best permutation is guided by parameter $\gm$. The traditional formulation is: {\sl Given a matrix $A$, find a permutation  $\pi$ and a Robinson matrix $R$ so that $\| A^{\pi}-R\| _p$ is minimized.} (Here $A^{\pi}$ refers to the matrix obtained by permuting rows and columns of $A$ according to $\pi$.) This approach requires the simultaneous optimization of both the matrix $R$ and the permutation $\pi$.
Using the results in this paper, we can instead reduce the noisy seriation problem to the following problem:

\smallskip\noindent
{\sl Given a symmetric matrix $A$, find a permutation of its rows and columns so that $\gm (A)$ is minimized.} 

\smallskip
Once such a permutation is found, our algorithm can be used to compute the appropriate Robinson approximation. While this approach is not an approximation algorithm per se, we do bound the performance of this approach in terms of the optimal outcome. Namely, as will be shown in Lemma \ref{lem:BestApprox}, the best possible Robinson approximation has normalized $\ell_1$ distance at least $\frac14 \gm (A)$ from $A$.  Therefore, our results implicitly bound the Robinson approximation achieved by our algorithm in terms of the optimal solution.

The results of this article are fundamentally different from any previous results on $\ell^\infty$-fitting Robinsonian structures (see for example  \cite{chepoi-fichet}). Indeed, when matrices grow large in size, the $\ell^1$-norm provides us with a more suitable notion of ``closeness''.  This fact becomes apparent when we analyze a growing sequence of graphs which are convergent in the sense of Lov\'{a}sz-Szegedy \cite{lovasz-szegedy}. Indeed, this article (and the  choice of notation for the parameter $\gm$) was motivated by our previous work \cite{linearembeddings}, where we introduce a parameter $\Gamma$ which characterizes Robinson \emph{graphons}. Graphons are symmetric functions on $[0,1]^2$ with values in $[0,1]$, which can be thought of as the ``blueprint'' of a random graph whose vertices are randomly sampled from the interval $[0,1]$. A matrix $A=[A_{i,j}]$ can be interpreted as a graphon in a natural way, by splitting the unit square  into  subsquares of size $\frac{1}{n}\times \frac{1}{n}$, and setting the graphon equal to $A_{i,j}$ everywhere in the $(i,j)$-th subsquare. A \emph{Robinson graphon} is a graphon  which is non-decreasing along every horizontal or vertical line towards the main diagonal. 

Our main result in \cite{linearembeddings} is that $\Gamma$ becomes a continuous parameter, when the space of graphons is equipped with the box-norm. Therefore $\Gamma$ provides us with a parameter to measure Robinson resemblance, which can be efficiently approximated. The parameters $\Gamma_1$ and $\Gamma$ are closely related, even though box-norm continuity does  not hold for $\Gamma_1$ anymore. In future work \cite{GhandJanss2019}, we employ   these parameters simultaneously in order to develop (continuous) methods for seriation of noisy data. 

Finally, we mention applications to graphs and networks. 
Binary Robinson matrices correspond to unit interval graphs. More precisely, a graph is a unit interval graph if and only if the adjacency matrix is Robinsonian, that is, there exists a labelling of the vertices so that the resulting adjacency matrix is a Robinson matrix. The parameter $\gm$ can be directly applied to a (labelled) graph, and if $\gm$ is sufficiently small, our algorithm constructs  a unit interval graph that is close, in edit distance, to the original  graph. 

When applied to real-life networks, the parameter $\gm$ can be used to measure how closely the matrix conforms to a linear model. In previous work, the authors investigated this question in the context of graph limits \cite{linearembeddings,uniform}. In a  \emph{linear graph model} the vertices of the graph are placed on a line, and the links are formed stochastically so that vertices that are closer together are more likely to connect. The linear layout can refer to a time line, in case of graphs derived from archeology, or the food chain, in case of food webs. But a linear layout may also point to the presence of a strong hidden variable that influences link formation, such as hierarchy, in a professional social network, or age in a friendship graph. If a graph conforms to a linear model,  then we expect the adjacency matrix of the graph to be ``almost-Robinson.'' Parameter $\gm$ may therefore serve as a measure of the ``linearity" of a given network.   

The rest of this article is organized as follows. In Section \ref{sec:def}, we introduce the necessary notations and definitions, and we state our main result, namely Theorem \ref{thm:main}. In Section \ref{sec:Approximate-binary}, we present Algorithm \ref{Alg:01} which finds a Robinson approximation for the special case where $A$ is a binary matrix ({\it i.e.} a graph adjacency matrix). The values of every cell in the Robinson approximation is decided based on the entries of $A$ in the {upper right} and the {lower left} regions defined by that cell (see Figure \ref{fig:UL}).  Algorithm \ref{Alg:01} is very simple to state, however one needs to use careful approximations and counting tricks to prove that the algorithm generates an output which indeed is a good $\ell^1$-approximation for the input matrix.  Section \ref{sec:general} provides us with an adaptation of Algorithm \ref{Alg:01} to general matrices. This is indeed a natural generalization, as every matrix with entries in $[0,1]$ decomposes into a convex combination of binary matrices. Fortunately, the parameter $\gm$ distributes over such decompositions, even though it is not a linear parameter in general. This allows us to apply Algorithm \ref{Alg:01} to the components of the decomposition in stages.  
In Section \ref{sec:improve}, we develop an algorithm that represents a preprocessing step for Algorithm \ref{Alg:01}. This preprocessing step is designed to transform the input matrix $A$, so that  Algorithm  \ref{Alg:01} generates a better-approximating output matrix $R$.
The trade-off here is that the preprocessing step increases the complexity of the algorithm, which still remains polynomial. 
We finish the paper by some concluding remarks and future directions.

\section{Definitions and main results}\label{sec:def}
For a given positive integer $n$, let ${\mathcal A}_n$ denote the set of all symmetric $n\times n$ matrices with entries in $[0,1]$.
Note that the restriction on the range of the entries is not a limitation, since it can always be achieved by shifting and scaling of the matrix. 
A matrix $A\in {\mathcal A}_n$ is called \emph{binary} if it has entries only from $\{0,1\}$. We refer to the position in the $i$-th row and $j$-th column of $A$ as the $(i,j)$'th \emph{cell}. For a matrix $A$ of size $n$, we define its (normalized) $\ell^1$-norm to be 
$\|A\|_1=\frac{1}{n^2}\sum_{i,j=1}^n |A_{i,j}|$. 
\begin{definition}
An $n\times n$ symmetric  matrix $A$ is a \emph{Robinson matrix} if, for all $1\leq i<j<k\leq n$, 
\begin{equation}
\label{Rprop}
A_{i,j}\geq A_{i,k}\mbox{ and }A_{j,k}\geq A_{i,k}.
\end{equation}
\end{definition}
In this section, we define a parameter, denoted by $\gm$, which measures how badly a matrix fails to be Robinson. 
The choice of notation for $\gm$ is due to the fact that it simply adds the magnitude of violations to (\ref{Rprop}). 
Precisely, given a symmetric matrix $A$ of size $n$,
\[
\gm (A) = \frac{1}{n^3}\sum_{1\leq i<k<j\leq n} [A_{i,j}-A_{i,k}]_+\,+[A_{i,j}-A_{k,j}]_+,
\]
where $[x]_+=x$ if $x\geq 0$, and 0 otherwise.
It is clear that $\gm(A)=0$ if and only if $A$ is a Robinson matrix. Note also that, for a binary matrix $A$ which is not Robinson, $\gm (A)\geq \frac{1}{n^3}$. Moreover,  the computation 
of $\gm (A)$ for an $n\times n$ matrix $A$ involves a simple summation which can be executed in $O(n^3)$ steps. Finally note that, due to the normalization factor $\frac{1}{n^3}$, $\gm (A)\in [0,1)$ whenever $A\in {\mathcal A}_n$.

We are now ready to state our main result (Theorem \ref{thm:main}), whose proof takes a large part of the paper and finishes only at the end of Section \ref{sec:improve}.

\begin{theorem}
\label{thm:main}
For every $A\in{\mathcal A}_n$, there exists a Robinson matrix $R\in {\mathcal A}_n$ so that
\[
\|A-R\|_1\leq \blue{26}\, \gm(A)^{1/3}.
\]
Moreover, $R$ can be computed in polynomial time. In addition, if $A$ is binary, then
there exists a binary matrix $R$ satisfying the conditions of the theorem.
\end{theorem}

The distance between $A$ and the Robinson  approximation obtained in Theorem \ref{thm:main} is bounded in terms of $\gm (A)$, and thus the algorithm does not necessarily give the best possible approximation. However, as stated in the following simple lemma, there is a close relationship between $\gm (A)$ and the distance between $A$ and the best possible Robinson approximation. 

\begin{lemma}\label{lem:BestApprox}
For every pair of matrices $A$ and $R$ in ${\mathcal A}_n$, if $R$ is a Robinson matrix then
$$\| A-R\|_1\geq \frac14 \gm (A).$$
Consequently, we have
$$\min\left\{\| A-R\|_1: R\in {\mathcal A}_n \mbox{ is Robinson similarity}\right\}\geq \frac14 \gm (A).$$
\end{lemma}

\begin{proof}
Since the function $[\, \cdot\, ]_+$ is sub-additive, we observe that 
$$\gm(A)=\gm(A-R+R)\leq \gm(A-R)+\gm(R)=\gm(A-R),$$
as $R$ is a Robinson matrix and thus $\gm(R)=0$. Moreover, for every symmetric matrix $B$ of size $n$, we have 
\begin{eqnarray*}
\gm (B) &=& \frac{1}{n^3}\sum_{1\leq i<k<j\leq n} [B_{i,j}-B_{i,k}]_+\,+[B_{i,j}-B_{k,j}]_+\\
&\leq& \frac{1}{n^3}\sum_{1\leq i<k<j\leq n} |B_{i,j}|+|B_{i,k}|+|B_{i,j}|+|B_{k,j}|\\
&\leq& \frac{4}{n^3}\sum_{k=1}^n\sum_{1\leq i,j\leq n} |B_{i,j}|\leq 4\|B\|_1.\\
\end{eqnarray*}
Letting $B=A-R$ finishes the proof. 
\end{proof}

Applying Theorem \ref{thm:main} to binary matrices, we obtain an interesting corollary for graphs. For a graph $G$, let the \emph{augmented adjacency matrix} $B_G$ denote the matrix which is obtained from the adjacency matrix of $G$ by replacing its diagonal entries by 1. Then, $B_G$ is a Robinson matrix precisely when  the 1-entries of each row and each column are all consecutive. This is called the \emph{symmetric consecutive ones property} ({\it i.e.} C1P for rows and columns simultaneously), and it is known to  characterize \emph{unit interval graphs}, or equivalently as shown in \cite{Roberts1969}, \emph{proper interval graphs} (see \cite{Corneil2004,Corneil1995,Looges1994}). Precisely, a graph is a unit interval graph if and only if there exists a linear order on its vertices with respect to which $B_G$ has the consecutive ones property, {\it i.e.}~is a Robinson matrix. 

The parameter $\gm$ of the augmented adjacency matrix counts the number of triples $(i,j,k)$, $1\leq i<j<k\leq n$, for which  vertices $i$ and $k$ are adjacent, but vertex $j$ is not adjacent to either $i$ or $k$. On the other hand, the (unnormalized) $\ell^1$-distance $\|B_G-{B}_{\widehat{G}}\|_1$ between the augmented adjacency matrices $B_G$ and $B_{\widehat{G}}$ of two labeled graphs $G$ and $\widehat{G}$ of the same order corresponds to the edit distance between the graphs themselves. The  \emph{edit distance} between two graphs $G$ and $\widehat{G}$, denoted by $ed(G,\widehat{G})$, is the minimum number of edge deletions and edge additions that need to be performed on $G$ to transfer it to $\widehat{G}$. Applying these concepts, Theorem \ref{thm:main} directly leads to the following corollary.

\begin{corollary}
\label{Cor:Graphs}
For every graph $G$  on vertex set $V=\{ 1,2,\dots ,n\}$, there exists a unit interval graph $\widehat{G}$ on vertex set $V$ so that 
\[ 
\frac{ed(G, \widehat{G})}{n^2}\leq \blue{26}\, \gm(G)^{1/3}.
\]
\end{corollary}

\section{Robinson similarity approximation for binary matrices}
\label{sec:Approximate-binary}
In this section, we present an algorithm that finds a Robinson approximation for the special case where $A$ is a binary matrix, and can thus be
 interpreted as the adjacency matrix of a graph. The algorithm can be intuitively understood as follows. We divide all cells of the matrix into \emph{black} and \emph{white} cells, and convert all zeros in the black cells to ones, and all ones in the white cells to zeros. The black region is \emph{convex around the diagonal}, in the sense that, if a cell is black, then so are all other cells closer to the diagonal in the same row or column. {In other words, the binary matrix whose support is precisely the black region is a Robinson matrix, which is indeed the Robinson approximation that the algorithm returns} (See Figure \ref{fig:black}).

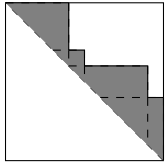
\begin{figure}
\begin{center}
\begin{tikzpicture}[scale=0.21]

\draw (0,0) -- (10,0) -- (10,10) -- (0,10) -- cycle;

\draw[dashed]  (10,0) -- (0,10);

\filldraw[fill=black!50, draw=black!50] (0,10) -- (4,10) -- (4,6) ;

\filldraw[fill=black!50, draw=black!50] (3,7) -- (5,7) -- (5,5) ;

\filldraw[fill=black!50, draw=black!50] (4,6) -- (9,6) -- (9,1) ;

\filldraw[fill=black!50, draw=black!50] (6,4) -- (10,4) -- (10,0) ;

\draw[dashed] (0,10) -- (4,10) -- (4,6) ;
\draw[dashed] (3,7) -- (5,7) -- (5,5) ;
\draw[dashed] (4,6) -- (9,6) -- (9,1) ;
\draw[dashed]  (6,4) -- (10,4) -- (10,0) ;

\draw (0,10) -- (4,10) -- (4,7) --(5,7) --(5,6)--(9,6)--(9,4)--(10,4)--(10,0);

\end{tikzpicture}
\end{center}
\caption{The black region is convex around the diagonal.}
\label{fig:black}
\end{figure}

The decision on whether to assign a cell to the black or white region depends on the entries in the \emph{upper right (UR)} and \emph{lower left (LL)} regions defined by the cell. 
Precisely, for any cell $(a,b)$, $1\leq a<b\leq n$,  we define
\begin{eqnarray*}
\UR(a,b)&=\{ (i,j):i<a<b<j\},\\
\LL(a,b)&=\{ (i,j):a\leq i\leq j\leq b\}.\\
\end{eqnarray*}
Roughly speaking, a cell will be black if it has enough ones in its upper right region, and it is white when it has 
enough zeros in its lower left region. So, we need the following notations (also shown in Figure \ref{fig:UL}):
\begin{eqnarray*}
1_{\UR}(a,b)&=&|\UR(a,b)\cap \{ (i,j):A_{ij}=1\}|,\\
0_{\LL}(a,b)&=&|\LL(a,b)\cap \{ (i,j):A_{ij}=0\}|.
\end{eqnarray*}
In addition, define $1_{\UR}(a,b)=0_{\LL}(a,b)=0$ when $a\in\{0,n+1\}$ or $b\in\{0,n+1\}$ or $a>b$. 
\begin{figure}[h]\label{fig:UL}
\begin{center}
\begin{tikzpicture}[scale=1.4]

\draw (0,0) -- (1.412,0) -- (1.412,-1.412) -- (0,-1.412) -- cycle;


\draw[dashed]  (0,0) -- (1.42,-1.42);

\draw (0, -.4) -- (1.412, -.4);

\draw (0, -.5) -- (1.412, -.5);


\draw (0.85, 0) -- (0.85, -.86); 




\draw (0.95, 0) -- (0.95, -0.94);

\filldraw[fill=blue!30, draw=blue!70!black] (0.95, -.4)  -- (1.412, -.4) -- (1.412, 0) -- (.95, 0);

\filldraw[fill=red!30, draw=red!70!black] (0.4, -.4)  -- (0.95, -.4) -- (.95, -0.94);

\draw (0, -.4) -- (1.412, -.4); 

\draw (0, -.5) -- (1.412, -.5);

\draw (0.85, 0) -- (0.85, -.86); 

\draw (0.95, 0) -- (0.95, -0.94); 

\filldraw[fill=green!30] (0.851, -0.398) -- (0.949, -0.398) -- (0.949, -0.498) -- (0.851, -0.498);

\node at (.87,.15) {\large $b$};


\node at (-.2,-.45) {\large $a$};


\end{tikzpicture}
\end{center}
\caption{Regions $\UR (a,b)$ (blue)  and $\LL (a,b)$ (red)}
\end{figure}
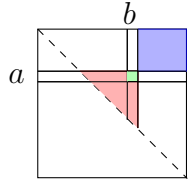


\begin{algorithm}
\caption{Robinson  approximation of a binary matrix\label{Alg:01}}
\begin{algorithmic}
\STATE{{\bf input:} binary matrix $A\in \An$, threshold \blue{$t>0$}}
\STATE{{\bf output:} binary Robinson  matrix $R\in \An$}
\FOR{$i\gets 1$ {\bf to} $n$}{
\blue{\STATE{$R_{1,i}\gets 0$; $R_{i,1}\gets 0$\;}
\STATE{$R_{i,n}\gets 0$; $R_{n,i}\gets 0$\;}}
}\ENDFOR
\FOR{$i \gets \blue{2}$ {\bf to} $n$}{
\STATE{ $j\gets \blue{n-1}$\;}
\WHILE{$j\geq i$}
\STATE{$1_{\UR}(i,j)\gets 1_{\UR}(i-1,j)+1_{\UR}(i,j+1)-1_{\UR}(i-1,j+1)+\blue{A_{i-1,j+1}}$\;}
\STATE{{\bf if} $1_{\UR}(i,j)<t$ {\bf then} $R_{i,j}\gets 0$; $R_{j,i}\gets 0$;}
\STATE{{\bf else} $R_{i,j}\gets 1$; $R_{j,i}\gets 1$}
\STATE{$j\gets j-1$;}
\ENDWHILE}
\ENDFOR
\RETURN $R$
\end{algorithmic}
\end{algorithm}

\medskip
Algorithm \ref{Alg:01} has complexity $\theta (n^2)$, and is therefore linear in the size of the input. A cell is considered black if $R_{i,j}$ is set to one, and white if $R_{i,j}$ is set to zero. If $(i,j)$ is black and 
$i\leq k<j$, then $\UR (i,j)\subset \UR (i,k)$, so $1_{\UR} (i,k)\geq 1_{\UR} (i,j)\geq t$, and thus $(i,k)$ is also black. Similarly, $(k,j)$ is also black. So, the region of black cells is convex around the diagonal, and $R$ is indeed a Robinson matrix. We will now show in Theorem \ref{thm:algo1} that the distance between $R$ and $A$ is bounded by a function of $t$ and $n$, \blue{if the parameter $t$ satisfies Condition (\ref{prop:URxLL}). 
We will then show in Corollary \ref{cor:algo1withbound} that an appropriate $t$ can be chosen as a function 
of $\gm (A)$, so that $\|A-R\|_1$ can be bounded in terms of $\gm (A)$.}

\begin{theorem}\label{thm:algo1}
Let $A\in\An$ be a binary matrix, with the property that 
\begin{equation}\label{prop:URxLL}
\text{for all }1\leq i\leq j\leq n,\, 1_\UR(i,j) < t\text{ or }0_\LL(i,j)< t.
\end{equation}
If $R$ is the matrix produced as output of Algorithm \ref{Alg:01} on input $A$  with threshold $t$, then $\|A-R\|_1\leq \frac{16\sqrt{t}+4}{n}$.
\end{theorem}

\begin{proof}
As explained earlier, black cells are exactly the cells $(i,j)$ for which $R_{i,j}$ attains  1, {\it i.e.} the cells for which $1_{\UR}(i,j)\geq t$.  Let $\cB$ denote the collection of all black cells above the diagonal, that is 
$$\cB=\{(i,j): 1\leq i\leq j\leq n \mbox{ and } 1_{\UR}(i,j)\geq t\}.$$
Note that ${\cB}$ is convex around the diagonal, in the sense that if a cell $(i,j)$ belongs to $\cB$ then $\LL(i,j)\subseteq \cB$. However, the black region can be disconnected. Precisely, there can be diagonal cells not in $\cB$, in which case $\cB$ consists of a collection of connected regions which are convex around the diagonal. Note also that, by definition, $\cB$ cannot contain any cells $(i,j)$ so that $|\UR(i,j)|<t$, and specifically, $\cB$ cannot contain any cells from the first row or the last column.

Let $\partial \cB$ denote the set of boundary cells of $\cB$, {\it i.e.}
$$\partial (\cB)=\{(i,j)\in \cB:\ (i-1,j)\not\in \cB\ \mbox{ or }\ (i,j+1)\not\in \cB\}.
$$
The set $\partial \cB$ precisely contains all black cells above the diagonal that are adjacent to cells outside $\cB$. Thus, 
\begin{equation}\label{eq-cover}
\cB = \bigcup_{(i,j)\in \partial \cB}\ \LL (i,j).
\end{equation}
Similarly, the white region $\cW$ and its boundary cells are defined as 
$$\cW=\{(i,j): 1\leq i\leq j\leq n \mbox{ and } 1_{\UR}(i,j)< t\},$$
and 
$$\partial \cW =\{(i,j)\in \cW:\  (i+1,j)\not\in \cW\ \mbox{ or }\ (i,j-1)\not\in \cW\}.$$
%


\medskip
\noindent{\bf Claim 1.} Let ${\cB}_0=\{(i,j)\in {\cB}:\ A_{i,j}=0\}$. Then $|{\cB}_0|\leq 4n\sqrt{t}$
\begin{proof}[Proof of Claim 1] 
First, list elements of $\partial \cB$ as $(i_1,j_1),(i_2,j_2),\ldots, (i_m, j_m)$ in such a way that  $i_1\leq i_2\leq \ldots\leq i_m$, and if $i_l=i_{l+1}$ then $j_l<j_{l+1}$. This ordering follows the ``contour" of the black region, starting with the first black cell in the first row containing any black cells. Since both indices range from at least 1 to at most $n$, it is clear that the boundary contains at most $2n$ cells, and $m\leq 2n$.

To control  the amount of possible overlaps in the covering of $\cB$ in (\ref{eq-cover}), we now  construct a subsequence ${\cal C}$ of $\partial\cB$, so that the lower left regions of cells in $\cal{C}$ cover most of the black region. 
The subsequence  ${\cal{C}}=\{ (i_{n_k},j_{n_k}) \}$ is constructed inductively. In the first step, \blue{let $( i_{n_1},j_{n_1})$ be the last cell in the first row of $\cB$.}  At step $k\geq 1$, let $n_k$ be the largest index so that $(i_{n_k},j_{n_k})\in \cC$. \blue{Define $n$ to be the first index in $\{1,\ldots, m\}$ which satisfies $j_{n_k}<j_{n}$, $(i_n,j_{n}+1)\not\in\cB$, and either $i_{n}-i_{n_k}> \lfloor\sqrt{t}\rfloor$  or $j_{n}-j_{n_k}> \lfloor\sqrt{t}\rfloor $.} So cell $(i_n,j_n)$ is the first cell in $\partial \cB$ whose row or column index differs by at least $\lfloor\sqrt {t}\rfloor +1$ from the last cell added to $\cal{C}$, and also is the last black cell in its row. Set $n_{k+1}=n$, and add $(i_{n_{k+1}},j_{n_{k+1}})$ to $\cC$. Since $\partial\cB$ has at most $2n$ elements, and $\lfloor \sqrt{t}\rfloor +1 \geq \sqrt{t}$, this inductive process ends in at most $\frac{2n}{\sqrt{t}}$ steps.  So $|\cC|\leq \frac{2n}{\sqrt{t}}$.

We now claim that $\cB\setminus\bigcup_{(i,j)\in \cC} \LL(i,j)$ can be covered with $|\cC|$ squares of dimensions $\lfloor\sqrt{t}\rfloor\times \lfloor\sqrt{t}\rfloor$. Consider two consecutive elements of $\cC$, say $(i_{n_k},j_{n_k})$ and $(i_{n_{k+1}},j_{n_{k+1}})$. \blue{Let $(i,j)$ be the first cell in $\partial \cB\setminus\bigcup_{(i,j)\in \cC} \LL(i,j)$ after $(i_{n_k},j_{n_k})$.}  By construction, $(i,j)$ is not in the same row or column as $(i_{n_k},j_{n_k})$, and thus $i>i_{n_k}$ and $j>j_{n_k}$. \blue{Let $(i',j')$ be the last cell before $(i_{n_{k+1}},j_{n_{k+1}})$ in $\partial \cB\setminus\bigcup_{(i,j)\in \cC} \LL(i,j)$}. (If no such $n$ exists, then $k$ is the last index in $\cal{C}$, and we let $(i',j')=(i_m,j_m)$, the last cell of $\partial \cB$.)

By construction of 
$\cC$, we know that $i'-i+1\leq i'-i_{n_k}\leq  \lfloor\sqrt{t}\rfloor$ and $j'-j+1\leq j'-j_{n_k}\leq \lfloor\sqrt{t}\rfloor$. Moreover, since $(i,j)$ and $(i',j')$ belong to $\partial \cB$, and $\cB$ is convex around the diagonal, every cell $(a,b)\in\cB$
with  $i\leq a\leq i'$ must satisfy $j\leq b\leq j'$,  so it must belong to the 
square whose top-left corner is $(i,j)$, and its bottom-right corner is $(i',j')$. We denote this square by $S_{k}$, and note that $| S_{k}| \leq t$.
Thus,
$$\cB\subseteq \bigcup_{k=1}^{|\cal{C}|} \LL(i_{n_k},j_{n_k})\cup\bigcup_{k=1}^{|\cal{C}|} S_{k}.$$
By Condition (\ref{prop:URxLL}), every cell $(i,j)\in \cB$ satisfies $0_{\LL}(i,j)\leq t$. So,
\[
|\cB_0|\leq \sum_{(i,j)\in \cC} 0_{\LL}(i,j)+\sum_{(i,j)\in \cC} |S_{i,j}|
 \leq | \cC | t + | \cC | t= 4n\sqrt{t}.
\]
\end{proof}

\noindent{\bf Claim 2.} Let $\cW_1=\{(i,j)\in\cW: \ A_{i,j}=1\}$. Then $|\cW_1|\leq  4n\sqrt{t}+2n$.
\begin{proof}[Proof of Claim 2]
This proof is similar to the proof of the previous claim. Cells in the boundary $\partial \cW$ are enumerated similarly, and  a subsequence $\cD=\{(i_{n_k},j_{n_k})\}$ of $\partial \cW$ is defined analogous to $\cC$: \blue{let $( i_{n_1},j_{n_1})$ be the last cell in the first column of $\cW$}, and for $k\geq 1$, define $n_{k+1}$ to be the first index in $\{1,\ldots, m\}$ \blue{which satisfies  $i_{n_{k+1}}>i_{n_k}$, $(i_{n_{k+1}}+1,j_{n_{k+1}})\not\in\cW$,} and  either $i_{n_{k+1}}-i_{n_k}> \lfloor\sqrt{t}\rfloor$ or $j_{n_{k+1}}-j_{n_k}> \lfloor\sqrt{t}\rfloor $. Let $S_k$ be the square of size at most $\lfloor\sqrt{t}\rfloor \times \lfloor\sqrt{t}\rfloor$ covering all white cells between consecutive cells of $\cal{D}$. Then we have

$$\cW\subseteq \partial \cW \cup \bigcup_{(i,j)\in \cD} \UR(i,j)\cup\bigcup_{k=1}^{|\cal{D}|} S_{k}$$
Note that  dealing with the white region is slightly different from the black region, in the sense that $\UR (i,j)$ does not include cell $(i,j)$ and the cells in row $i$ and column $j$ (see Figure \ref{fig:UL}). So,  we include the boundary explicitly to cover the white region.

%
By definition of $\cW$, every white cell  $(i,j)\in \cW$ satisfies $1_{\UR}(i,j)<{t}$. This implies that 
\[
|\cW_1|\leq  |\partial\cW |+ \sum_{(i,j)\in \cD} 1_{\UR}(i,j)+\sum_{i=1}^{|\cD|} |S_{k}|\leq 2n+|\cD | t+ |\cD| t\leq 4n\sqrt{t}+2n.
\]
\end{proof}

The above two claims show that $R$ and $A$ differ above the diagonal in at most $8n\sqrt{t} +2n$ cells. Adding the region below the diagonal and normalizing, we conclude that
$\|A-R\|_1\leq  \frac{2(8\sqrt{t}+2)}{n}$.
\end{proof}

The following simple counting lemma provides a threshold, in terms of $\gm$, for which Condition (\ref{prop:URxLL}) always holds. 
\begin{lemma}\label{lem:0-1count}
Let $A\in {\cal A}_n$ be a binary matrix {whose diagonal entries are all 1. }
Then, for every cell $(i,j)$,
we have 
$$1_\UR(i,j) \ 0_\LL(i,j)\leq 2n^4\gm(A).$$
\end{lemma}
\begin{proof}
Without loss of generality, assume that $i<j$.
\begin{eqnarray*}
1_\UR(i,j) \ 0_\LL(i,j)&=&\sum_{{\scriptsize
\begin{array}{l}
(s,t)\in \UR (i,j) \\
(s',t')\in \LL (i,j)
\end{array}
}}[A_{s,t}-A_{s',t'}]_+\\
&\leq&\frac{1}{2}\sum_{{\scriptsize
1\leq s< i\leq s'\leq t'\leq j< t \leq n
}}  [A_{s,t}-A_{s,t'}]_++[A_{s,t'}-A_{s',t'}]_+ + [A_{s,t}-A_{s',t}]_++[A_{s',t}-A_{s',t'}]_+\\
&\leq &\frac{1}{2}\sum_{{\scriptsize
1\leq s< s'\leq t'< t \leq n
}}  [A_{s,t}-A_{s,t'}]_++ [A_{s,t}-A_{s',t}]_+\\
&+&\frac{1}{2}\sum_{{\scriptsize
1\leq s<s'\leq t'< t \leq n
}} [A_{s,t'}-A_{s',t'}]_+ +[A_{s',t}-A_{s',t'}]_+,\\
\end{eqnarray*}
where the second inequality can be justified using the fact that if  $[A_{s,t}-A_{s',t'}]_+=1$ then both $[A_{s,t}-A_{s,t'}]_++[A_{s,t'}-A_{s',t'}]_+=1$ and  $[A_{s,t}-A_{s',t}]_++[A_{s',t}-A_{s',t'}]_+=1$ are satisfied.
Now observe that 
$$\sum_{{\scriptsize
1\leq s< s'\leq t'< t \leq n
}}  [A_{s,t}-A_{s,t'}]_+\leq \sum_{s'=1}^n\sum_{{\scriptsize
1\leq s< t'< t \leq n
}}  [A_{s,t}-A_{s,t'}]_+\leq n^4\gm(A),$$
and similarly $\sum_{{\scriptsize
1\leq s< s'\leq t'< t \leq n
}}  [A_{s,t}-A_{s',t}]_+\leq n^4\gm(A).$ Moreover, for every $s$ and $s'$, we have $ [A_{s,s'}-A_{s',s'}]_+=0$, since $A_{s',s'}=1$. Thus, 
\begin{eqnarray*}
\sum_{{\scriptsize
1\leq s<s'\leq t'< t \leq n
}} [A_{s,t'}-A_{s',t'}]_+&=&\sum_{{\scriptsize
1\leq s<s'< t'< t \leq n
}} [A_{s,t'}-A_{s',t'}]_+\\\
&\leq& \sum_{t=1}^n\sum_{{\scriptsize
1\leq s<s'< t'\leq n
}} [A_{s,t'}-A_{s',t'}]_+\leq n^4\gm(A), 
\end{eqnarray*}
and similarly $\sum_{{\scriptsize
1\leq s<s'\leq t'< t \leq n
}}[A_{s',t}-A_{s',t'}]_+\leq n^4\gm(A).$ This finishes the proof. 
\end{proof}
\begin{corollary}\label{cor:algo1withbound}
For every binary matrix $A\in {\mathcal A}_n$, there exists a binary Robinson matrix $R\in {\mathcal A}_n$ such that 
$$\|A-R\|_1\leq {\frac{5}{n}+2^{9/2}\gm(A)^{1/4}}.$$
Moreover, $R$ can be computed in linear time.
\end{corollary}
\begin{proof}
{From the binary matrix $A\in {\mathcal A}_n$,  we first construct $\widetilde{A}$ by replacing every 0 entry on the diagonal of $A$ with 1. Clearly, $\widetilde{A}\in {\mathcal A}_n$ and $\gm(\widetilde{A})\leq \gm(A)$. Moreover, $\|A-\widetilde{A}\|_1\leq \frac{1}{n}$. If $\gm(\widetilde{A})=0$, we take $R=\widetilde{A}$ and we are done. So assume that $\gm(\widetilde{A})>0$. 
Let   $R\in {\mathcal A}_n$ be the matrix produced as output of Algorithm \ref{Alg:01} on input $\widetilde{A}$  with threshold $t=n^2\sqrt{4\gm(\widetilde{A})}$.
By Lemma \ref{lem:0-1count} and the fact that $\gm(\widetilde{A})>0$, $\widetilde{A}$ and parameter $t=n^2\sqrt{4\gm(\widetilde{A})}$ satisfy Condition (\ref{prop:URxLL}) of Theorem 
\ref{thm:algo1}. Thus,
$$\|R-A\|_1\leq \|A-\widetilde{A}\|_1+\|\widetilde{A}-R\|_1\leq \frac{5}{n}+16(4\gm(\widetilde{A}))^{1/4}\leq \frac{5}{n}+2^{9/2}\gm(A)^{1/4}.$$ }
\end{proof}
\section{Robinson similarity approximations of general matrices}\label{sec:general}
For general matrices, we first decompose the matrix into a convex combination of binary matrices, a standard technique widely used in the literature, and then apply Algorithm \ref{Alg:01} to each summand. Given any matrix $A\in \cA_n$, let ${\rm range}(A)= \{ A_{i,j}:1\leq i\leq j\leq n\}$. Consider the linear ordering $0=s_0<s_1<\dots <s_m$ of  ${\rm range}(A)\cup \{ 0\}$, and define matrices $A^{(k)}$, $1\leq k\leq m$ as follows:

\begin{equation}\label{eq:layer}
A_{i,j}^{(k)}=
\left\{ \begin{array}{ll}1\ & \text{if }A_{i,j}\geq s_k , \\ 0\ & \text{otherwise.}\end{array}\right.
\end{equation}
Clearly, each matrix $A^{(k)}$ is binary, and
\begin{equation}\label{eq:decomp}
A=\sum_{k =1}^m (s_k-s_{k-1})A^{(k )}.
\end{equation}
We refer to the matrices $A^{(k)}$ as the \emph{layers} of $A$. 
From the definition of the layers, it is easy to see that $A^{(k)}\leq A^{(l)}$ whenever $l<k$. Indeed for such $l$ and $k$, if $A^{(k)}_{i,j}=1$ then 
 $A^{(l)}_{i,j}=1$ as well, since $s_l<s_k$. 

{ 
The advantage of writing $A$ as a convex combination of its layers, as opposed to any other decomposition of $A$, lies in the fact that $\Gamma_1$ distributes over this particular decomposition of $A$, even though $\Gamma_1$ is not a linear map in general. }

\begin{proposition}
Suppose $A$ is ``layered'' as in Equation (\ref{eq:decomp}). Then we have
\begin{eqnarray}
\label{eqn:linear}
\gm (A)= \sum_{l=1}^m (s_l-s_{l-1}) \gm (A^{(l)}).
\end{eqnarray}
\end{proposition}
\begin{proof}
Fix a triple $i,j,k$ satisfying $1\leq i<k<j\leq n$, and let $n_1,n_2,n_3\in\{1,\ldots,m\}$ be so that 
$A_{i,j}=s_{n_1}$, $A_{i,k}=s_{n_2}$, and $A_{k,j}=s_{n_3}$. From the definition of the layers in (\ref{eq:layer}), we have 
\begin{itemize}
\item[(i)] $A_{i,j}^{(l)}=1$ if $l\leq n_1$, and $A_{i,j}^{(l)}=0$ if $l> n_1$.
\item[(ii)] $A_{i,k}^{(l)}=1$ if $l\leq n_2$, and $A_{i,k}^{(l)}=0$ if $l> n_2$.
\item[(iii)] $A_{k,j}^{(l)}=1$ if $l\leq n_3$, and $A_{k,j}^{(l)}=0$ if $l> n_3$.
\end{itemize}
Note first that $[A_{i,j}-A_{i,k}]_+=[s_{n_1}-s_{n_2}]_+=0$  precisely when $n_1\leq n_2$. Using (i), (ii) and (iii) it is easy to see that, if $n_1\leq n_2$, then $[A_{i,j}^{(l)}-A_{i,k}^{(l)}]_+=0$ for every $1\leq l\leq m$. On the other hand, if $[A_{i,j}-A_{i,k}]_+>0$ and thus $n_1> n_2$, then, if $l\leq n_2$ or $l>n_1$ then  $[A_{i,j}^{(l)}-A_{i,k}^{(l)}]_+=0$, and if $n_2<l\leq n_1$ then $[A_{i,j}^{(l)}-A_{i,k}^{(l)}]_+=1$. 
Hence, we can verify the following claim:
\begin{eqnarray*}\label{eq:claim-layer}
[A_{i,j}-A_{i,k}]_+=\sum_{l=1}^m (s_l-s_{l-1})[A_{i,j}^{(l)}-A_{i,k}^{(l)}]_+.
\end{eqnarray*}
Indeed, if $n_1\leq n_2$ then both sides of the above equation are equal to 0. For the case where $n_1>n_2$, we have
\begin{eqnarray*}
\sum_{l=1}^m(s_l-s_{l-1})[A_{i,j}^{(l)}-A_{i,k}^{(l)}]_+
=\sum_{l=n_2+1}^{n_1}s_l-s_{l-1}=s_{n_1}-s_{n_2}=[A_{i,j}-A_{i,k}]_+.
\end{eqnarray*}
Repeating the above argument, we obtain a similar claim for $[A_{i,j}^{(l)}-A_{k,j}^{(l)}]_+$, and consequently we get 
\begin{eqnarray*}
[A_{i,j}-A_{i,k}]_++[A_{i,j}-A_{k,j}]_+=\sum_{l=1}^m (s_l-s_{l-1})\Big([A_{i,j}^{(l)}-A_{i,k}^{(l)}]_++[A_{i,j}^{(l)}-A_{k,j}^{(l)}]_+\Big).
\end{eqnarray*}
So we have  
\begin{eqnarray*}
\sum_{l=1}^m (s_l-s_{l-1}) \gm (A^{(l)})&=& \sum_{l=1}^m (s_l-s_{l-1})\Big(\frac{1}{n^3}\sum_{1\leq i<k<j\leq n} [A^{(l)}_{i,j}-A^{(l)}_{i,k}]_+\,+[A^{(l)}_{i,j}-A^{(l)}_{k,j}]_+\Big)\\
&=& \frac{1}{n^3}\sum_{1\leq i<k<j\leq n} \sum_{l=1}^m (s_l-s_{l-1})\Big([A^{(l)}_{i,j}-A^{(l)}_{i,k}]_+\,+[A^{(l)}_{i,j}-A^{(l)}_{k,j}]_+\Big)\\
&=& \frac{1}{n^3}\sum_{1\leq i<k<j\leq n} [A_{i,j}-A_{i,k}]_+\,+[A_{i,j}-A_{k,j}]_+=\gm(A),
\end{eqnarray*}
which finishes the proof. 
\end{proof}

\begin{algorithm}
\caption{Robinson similarity approximation of a general matrix\label{Alg:general}}
\begin{algorithmic}
\STATE{{\bf input:} Matrix $A\in \An$, \blue{positive} thresholds $t_1,\dots ,t_m$}
\STATE{{\bf output:} Robinson similarity matrix $R\in \An$}
\STATE{Compute ${\rm range}(A)\cup \{ 0,1\}$, as an ordered list $s$\;}
\FOR{$i\gets 1$ {\bf to} $n$}{
\FOR{$k\gets 1$ {\bf to}  $m$}{
\blue{\STATE{$R^{(k)}_{1,i}\gets 0;$ $R^{(k)}_{i,1}\gets 0$\;}
\STATE{$R^{(k)}_{i,n}\gets 0;$ $R^{(k)}_{n,i}\gets 0$\;}
\STATE{$A_{i,i}\gets 1$ \;}}
}\ENDFOR
}\ENDFOR
\FOR{$i \gets \blue{2}$ {\bf to} $n$}{ 
\STATE{$j\gets \blue{n-1}$\;}
\WHILE{$j\geq i$}{
\FOR{$k\gets 1$ {\bf to}  $m$}{
\STATE{{\bf if} $A_{i,j}\geq s[k]$ {\bf then} $temp\gets 1$;}
\STATE{{\bf else} $temp\gets 0$}
\STATE{ $1^k_{\UR}(i,j)\gets 1^k_{\UR}(i-1,j)+1^k_{\UR}(i,j+1)-1^k_{\UR}(i-1,j+1)+temp$\;}
\STATE{{\bf if} $1^k_{\UR}(i,j)<t_k$ {\bf then}  $R^{(k)}_{i,j}\gets 0;$ $R^{(k)}_{j,i}\gets 0$;}
\STATE{{\bf else}   $R^{(k)}_{i,j}\gets 1;$ $R^{(k)}_{j,i}\gets 1$}}
\ENDFOR
\STATE{$j\gets j-1$;}}
\ENDWHILE}
\ENDFOR
\STATE{$R\gets {\mathbf{0}}$\;}
\STATE{{\bf for} $k\gets 1$ {\bf to} $m$ {\bf do} $R\gets R+(s[k]-s[k-1])R^{(k)}$}
\RETURN $R$
\end{algorithmic}
\end{algorithm}

\medskip
Algorithm \ref{Alg:general} is essentially a simultaneous execution of Algorithm \ref{Alg:01} for each binary matrix $A^{(k)}$. 
The quantities $1_{\UR}^k(i,j)$ thus refer to the number of ones in the region $\UR (i,j)$ in $A^{(k)}$. The algorithm has complexity $O(n^2m)$, where $m$ is the size of ${\rm range}(A)$, and thus $m\leq n^2$. Since every matrix $R^{(k)}$ is Robinson similarity, so is their linear combination $R$. In the following theorem, we will show that there exist thresholds $t_1,\dots ,t_m$ so that the difference between $A$ and $R$ is bounded by $C \gm (A)^{1/4}$ for some constant $C$. To avoid anomalous behavior, we assume that our input matrix is not Robinson similarity, {\it i.e.}~$\gm (A)>0$.

\begin{theorem}
\label{thm:algo2}
Let $A\in \cA_n$, and  $\gm (A)>0$.
If $R$ is the matrix produced as output of Algorithm \ref{Alg:general} on input $A$  with thresholds $t_k=\sqrt{4\gm \pink{(A^{(k)})}}n^2$ for $k=1,\dots ,m$, then 
\[\|A-R\|_1\leq \blue{2^{9/2}}{\gm(A)}^{1/4}(1+O(n^{-1/4})).\]
\end{theorem}

\begin{proof}
First suppose $A$ is a binary matrix. \pink{Then by Corollary \ref{cor:algo1withbound} applied to $A$ with threshold $\sqrt{4\gm \pink{(A)}}n^2$,  we get \blue{$\|A-R\|_1\leq {\frac{5}{n}+2^{9/2}\gm(A)^{1/4}}.$ }}
%

Next, assume that $A\in \cA_n$ is a general matrix, and let $A=\sum_{k =1}^m (s_k-s_{k-1})A^{(k )}$ be the decomposition of $A$ into layers of binary matrices as described in Equation (\ref{eq:decomp}). Then,  we have $\gm(A)=\sum_{k=1}^m (s_k-s_{k-1})\gm(A^{(k)}).$ To avoid clutter of notation, let $\epsilon=\gm(A)$ and  $\epsilon_k:=\gm(A^{(k)})$. For every $1\leq k\leq m$, we apply \blue{Corollary \ref{cor:algo1withbound} to $A^{(k)}$ with threshold  
$t_k=(4\epsilon_k)^{1/2}n^2$} to obtain a Robinson similarity matrix $R^{(k)}$ such that
$$\|A^{(k)}-R^{(k)}\|_1\leq \blue{\frac{5}{n}+2^{9/2}\epsilon_k^{1/4}}.$$ 
So, Algorithm \ref{Alg:general} computes $R=\sum_{k=1}^m (s_k-s_{k-1}) R^{(k)}$, which is Robinson similarity as well. Moreover, 
\begin{eqnarray*}
\|A-R\|_1&\leq &\sum_{k=1}^m(s_k-s_{k-1})\|A^{(k)}-R^{(k)}\|_1\\
&\leq & \blue{2^{9/2}} \sum_{k=1}^m(s_k-s_{k-1})\epsilon_k^{1/4}+\sum_{k=1}^m(s_k-s_{k-1})\frac{\blue{5}}{n}\\
&\leq & \blue{2^{9/2}} (\gm(A))^{1/4}+\frac{\blue{5}}{n},\\
\end{eqnarray*}
where in the last inequality we used the fact that the function $f(x)=x^{1/4}$ is concave. 

By definition, if $A\in \cA_n$ is not Robinson similarity, then $\gm(A)\geq \frac{1}{n^3}$, and thus $n\gm(G)^{1/4}\geq n^{1/4}$. This completes the proof. 
\end{proof}

\section{Improvement through preprocessing}\label{sec:improve}
Finally, we give an algorithm that represents a preprocessing step for Algorithm \ref{Alg:01}. By Theorem \ref{Alg:01}, the Robinson similarity approximation produced by Algorithm  \ref{Alg:01} is at bounded distance from the input matrix $A$, provided  that Condition (\ref{prop:URxLL}) holds. By Lemma \ref{lem:0-1count}, the condition holds for every matrix, if we choose $t>\sqrt{2}n^2(\gm (A))^{1/2}$ (see proof of Corollary \ref{cor:algo1withbound}.) The preprocessing step is designed to transform the input matrix $A$, such that the condition holds for a smaller value of $t$. The modified matrix is then used as input to Algorithm  \ref{Alg:01} with the new value of $t$, which leads to an output matrix $R$ that is closer to the input matrix. 


The trade-off here is that the preprocessing step increases the complexity of the algorithm. This increase is tolerable, as the complexity still remains polynomial. We give an implementation which is quadratic for binary matrices, but we believe that a more sophisticated implementation would lead to an improvement in the complexity. In the following, the algorithm is given in detail for binary matrices only. It can easily be adapted to general matrices, in much the same way that Algorithm \ref{Alg:general} is adapted from Algorithm \ref{Alg:01}. 

First, we introduce some terminology. Suppose that  matrix $A$ is given, and a threshold value $t>0$ is fixed. \blue{ Let $\Delta$ denote the collection of all cells in $A$ which are on or above the diagonal.} We call a cell $(i,j)\in \Delta$ \emph{inverted} if $1_{\UR}(i,j)\geq t$ and $0_{\LL}(i,j)\geq t$. Thus, there exist inverted cells if and only if Condition (\ref{prop:URxLL}) is violated. To \emph{toggle} a cell $(i,j)$ is to set the value of all cells in $\UR (i,j)$ equal to zero, and the values of all cells in $\LL (i,j)$ equal to 1. It is easy, yet important, to observe  that toggling a cell can only decrease $\gm (A)$. 
{
\begin{lemma}\label{lem:toggle-decreases}
Let $A\in \An$ and $t>0$ be given. Suppose that $(i,j)\in \Delta$ is an inverted cell, and let $\widetilde{A}$ denote the matrix obtained from $A$ by toggling $(i,j)$. Then \blue{for any fixed triple $1\leq s<k<t\leq n$, we have
\begin{equation}\label{eq:toggle-decrease}
[\widetilde{A}_{s,t}-\widetilde{A}_{s,k}]_+\leq [A_{s,t}-A_{s,k}]_+ \ \mbox{ and } \ [\widetilde{A}_{s,t}-\widetilde{A}_{k,t}]_+\leq [A_{s,t}-A_{k,t}]_+.
\end{equation}
Consequently, we get $\gm(\widetilde{A})\leq \gm(A)$. }
\end{lemma}
\begin{proof}
\blue{
We only prove the first inequality; the second one can be proved similarly. Fix a triple $1\leq s<k<t\leq n$. 
First note that, the first inequality in (\ref{eq:toggle-decrease}) holds trivially whenever} $(s,t), (s,k)\in \Delta\setminus(\LL(i,j)\cup\UR(i,j))$, as on these cells the matrices $A$ and $\widetilde{A}$ are identical. 
For the cases where $(s,t)\in \UR(i,j)$ or $(s,k)\in \LL(i,j)$, we have $\widetilde{A}(s,t)=0$ or $\widetilde{A}(s,k)=1$. Thus, in both cases, we get  $[\widetilde{A}_{s,t}-\widetilde{A}_{s,k}]_+=0$, and in particular  $[\widetilde{A}_{s,t}-\widetilde{A}_{s,k}]_+\leq [A_{s,t}-A_{s,k}]_+$. Moreover, it is clear from the definition of the region $\LL(i,j)$ that if $(s,t)\in \LL(i,j)$ then $(s,k)\in \LL(i,j)$. Similarly, if $(s,k)\in \UR(i,j)$ then $(s,t)\in \UR(i,j)$ as well. Putting all these together, we conclude that the desired inequality holds in all cases. 
 Therefore, 
\begin{eqnarray*}
\gm (\widetilde{A}) &=& \frac{1}{n^3}\sum_{1\leq s<k<t\leq n} [\widetilde{A}_{s,t}-\widetilde{A}_{s,k}]_+\,+[\widetilde{A}_{s,t}-\widetilde{A}_{k,t}]_+\\
&\leq&\frac{1}{n^3}\sum_{1\leq s<k<t\leq n} [{A}_{s,t}-{A}_{s,k}]_+\,+[{A}_{s,t}-{A}_{k,t}]_+\\
&=&\gm(A),
\end{eqnarray*}
which finishes the proof. 
\end{proof}
}

Algorithm \ref{Alg:pre} can then be described as follows: each inverted cell is toggled, and after each toggling step, the values of $1_{\UR} $ and $0_{\LL}$ are recalculated for each cell. 

\begin{algorithm}
\caption{Preprocessing step\label{Alg:pre}}
\begin{algorithmic}
\STATE{{\bf input:} Matrix $A\in \An$, threshold $t$}
\STATE{{\bf output:} Updated matrix $A$}
\FOR{$i \gets 1$ {\bf to} $n$}{
\FOR{$j \gets i$ {\bf to} $n$}{
\STATE{Compute $1_{\UR}(i,j)$ and $0_{\LL}(i,j)$\;}
\STATE{{\bf if} $1_{\UR}(i,j)\geq t$ {\bf and} $0_{\LL}(i,j)\geq t$ {\bf then}}
\STATE{{\bf for} all cells $(r,s)\in \UR (i,j)$ {\bf do } $A_{r,s}\gets 0;\,A_{s,r}\gets 0$}
\STATE{{\bf for} all cells $(r,s)\in \LL (i,j)$ {\bf do } $A_{r,s}\gets 1;\,A_{s,r}\gets 1$}
}
\ENDFOR
}
\ENDFOR
\end{algorithmic}
\end{algorithm}

\medskip
Note that Algorithm \ref{Alg:pre} involves computing  $1_{\UR}(a,b)$ for all $1\leq a\leq i$ and $j\leq b\leq n$, and computing $0_{\LL}(a,b)$ for all $i\leq a \leq j$ and $i\leq b\leq j$. These values must be recomputed in each iteration, since $A$ is being changed. Any cell $(i,j)$ is tested exactly once, to see whether it is inverted, and if it is, it is toggled.  This naive implementation of the algorithm has complexity $O(n^4)$ and is thus quadratic in the size of the input. When adapted to general matrices, the complexity becomes $O(mn^4)$, where $m$ is the number of different values taken by entries of $A$. Clearly $m\leq n^2$. Also, if all entries of $A$ are rounded to the nearest multiple of $\epsilon^A{1/3}$ (recall that $\epsilon =\Gamma_1 (A)$), then $m=\epsilon^{-1/3}$, \blue{while the error between $A$ and $R$  is still of the same order.}

One may be concerned  that toggling some cells may create new inverted cells, in which case, just considering each cell once would not be sufficient, and the complexity could increase. The following lemma shows this cannot be the case. 

\begin{lemma}\label{lem:preproc}
Suppose Algorithm \ref{Alg:pre} is applied to a binary matrix $A$, with threshold $t$. Then  the output of the algorithm, which we call the modified matrix, satisfies the condition that, for each cell $(i,j)$, $1_{\UR}(i,j)<t$ or $0_{\LL}(i,j)<t$. 
\end{lemma}

\begin{proof}
It suffices to show that, if an inverted cell $(i,j)$ is toggled, no cell that was not inverted before the toggle can become inverted afterwards. Suppose to the contrary that there exists a cell $(k,\ell)$ which becomes inverted after the toggle. 
That is, either $1_{\UR}(k,\ell)$ is increased by the toggle, or $0_{\LL} (k,\ell)$ is {increased}. Assume without loss of generality  that $1_{\UR}(k,\ell)$ increases after the toggle. 
The toggle sets all entries in $\UR (i,j)$ to zero, and all entries to $\LL (i,j)$ to one. So for $1_{\UR}(k,\ell)$ to increase, $\UR (k,\ell)$ and $\LL (i,j)$ must intersect. Thus, $(k,\ell)\in \LL (i,j)$. Therefore, after the toggle, all entries in $\LL (k,\ell)$ have been set to one, so {$0_{\LL}(k,\ell)=0<t$}. Therefore, $(k,\ell)$ is not an inverted cell after the toggle, which contradicts our assumption.
\end{proof}

The above lemma shows that the modified matrix $\widehat{A}$ returned by Algorithm \ref{Alg:pre} satisfies Condition \ref{prop:URxLL} of Theorem \ref{thm:algo1} for the chosen value of $t$. Thus, this matrix can be used as input for Algorithm \ref{Alg:01} to obtain a Robinson similarity approximation $R$, and Theorem \ref{thm:algo1} can be applied to bound $\| \widehat{A}-R\|_1$. However, to obtain a good approximation we need to make sure we can  also  bound  $\| A-\widehat{A}\|_1$, the distance between the input and output matrices of Algorithm \ref{Alg:pre}. The following lemma gives such a bound. 

\begin{lemma}\label{lemma:prep}
Let $A\in \cA_n$ be a binary matrix. Let $\widehat{A}$ denote the output of Algorithm \ref{Alg:pre} with threshold $t$. Then $\|A-\widehat{A}\|_1\leq \frac{4n^2}{t}(\gm(A)-\gm(\widehat{A}))$.
\end{lemma}
\begin{proof}
The output matrix  $\widehat{A}$ is obtained from $A$ by consecutive toggling of inverted cells, occurring in lines 5 and 6 of Algorithm \ref{Alg:pre}. The condition in line 4 checks whether a cell is inverted. Let $A=A^0, A^1,A^2,\ldots, A^m=\widehat{A}$ denote the matrices in the intermediate steps. We will bound the distance between two consecutive matrices.

Fix $s$, $0\leq s<m$, and assume that $A^s$ is modified because cell $(i,j)$ is found to be inverted, and is then toggled.
So $1_{\UR}(i,j)\geq t$ and $0_{\LL}(i,j)\geq t$, where $1_{\UR}$ and $0_{\LL}$ are computed from $A^s$. So $A^{s+1}$ is formed from $A^s$ by adjusting every cell in $\UR(i,j)$ and its counterpart below the diagonal to be 0, and every cell in $\LL(i,j)$ and its counterpart below the diagonal to be 1. 

We can intuitively observe that $\gm (A^s)$ drops by at least $ \frac{ 0_{LL}(i,j)1_{UR}(i,j)}{n^4}$ when $(i,j)$ is toggled. \blue{Note that, normalizing as in the definition of $\gm$ and applying the same reasoning as in Lemma \ref{lem:0-1count}, one may be tricked into thinking that $\gm (A^s)$ drops at least by $\frac{2}{n^3}0_{LL}(i,j)1_{UR}(i,j)$ when $(i,j)$ is toggled, since a contribution of $\frac{2}{n^3}$ is removed for each pair of ``bad'' cells from $\UR (i,j)$ and $\LL (i,j)$. However, this reasoning does not take overcounting into consideration. We dedicate the following claim to a rigorous proof of this intuitive observation. }

{
\begin{claim}
For $A$, $s$, and $(i,j)$ as above, we have
$$\gm(A^s)-\gm(A^{s+1})\geq \frac{0_{\LL}(i,j)1_{\UR}(i,j)}{n^4}.$$
\end{claim}
\begin{proof}[Proof of claim]
First, consider a cell $(k,\ell)$ from $\UR(i,j)$ containing 1,  and a cell $(k',\ell')$  from $\LL(i,j)$, containing 0. 
Similar to the proof of Lemma \ref{lem:0-1count}, 
we have that 
\begin{eqnarray}
2=2[A^s_{k,\ell}-A^s_{k',\ell'}]_+ &=& [A^s_{k,\ell}-A^s_{k',\ell}]_+ +[A^s_{k',\ell}-A^s_{k',\ell'}]_+\nonumber\\
&&+[A^s_{k,\ell}-A^s_{k,\ell'}]_+ +[A^s_{k,\ell'}-A^s_{k',\ell'}]_+. \label{equality-1}
\end{eqnarray}
\blue{By Lemma \ref{lem:toggle-decreases},} for every triple  $1\leq k<l'<l\leq n$ we have 
$$[{A^{s+1}}_{k,l}-{A^{s+1}}_{k,l'}]_+\leq [A^s_{k,l}-A^s_{k,l'}]_+ \ \mbox{ and } \ 
[{A^{s+1}}_{k,l}-{A^{s+1}}_{l',l}]_+\leq [{A^{s}}_{k,l}-{A^{s}}_{l',l}]_+.$$
This, together with the fact that $A^{s+1}_{k,l}=0$ whenever $(k,l)\in \UR(i,j)$, implies that
\begin{eqnarray*}
\sum_{\begin{array}{c}
 \vspace{-0.15cm}
 \scriptscriptstyle{1\leq k<l'<l\leq n}\\
 \vspace{-0.15cm}
 \scriptscriptstyle{(k,l)\in \UR(i,j)}\\
 \scriptscriptstyle{i\leq l'\leq j}
\end{array}} [A^s_{k,l}-A^s_{k,l'}]_+&=&
\sum_{\begin{array}{c}
 \vspace{-0.15cm}
  \scriptscriptstyle{1\leq k<l'<l\leq n}\\
   \vspace{-0.15cm}
 \scriptscriptstyle{(k,l)\in \UR(i,j)}\\
 \scriptscriptstyle{{i\leq l'\leq j}}
\end{array}} [A^s_{k,l}-A^s_{k,l'}]_+-[{A^{s+1}}_{k,l}-{A^{s+1}}_{k,l'}]_+\\
&\leq&
\sum_{1\leq k<l'<l\leq n} [A^s_{k,l}-A^s_{k,l'}]_+-[{A^{s+1}}_{k,l}-{A^{s+1}}_{k,l'}]_+.
\end{eqnarray*}
Similarly, we have 
\begin{eqnarray*}
\sum_{\begin{array}{c}
 \vspace{-0.15cm}
 \scriptscriptstyle{1\leq k<k'<l\leq n}\\
 \vspace{-0.15cm}
 \scriptscriptstyle{(k,l)\in \UR(i,j)}\\
 \scriptscriptstyle{i\leq k'\leq j}
\end{array}} [A^s_{k,l}-A^s_{k',l}]_+
&\leq&
\sum_{1\leq k<k'<l\leq n} [A^s_{k,l}-A^s_{k',l}]_+-[{A^{s+1}}_{k,l}-{A^{s+1}}_{k',l}]_+.
\end{eqnarray*}
Adding up the above two inequalities, we get,
\begin{equation}\label{eq:half1}
\sum_{\begin{array}{c}
 \vspace{-0.15cm}
 \scriptscriptstyle{1\leq k<l'<l\leq n}\\
 \vspace{-0.15cm}
 \scriptscriptstyle{(k,l)\in \UR(i,j)}\\
 \scriptscriptstyle{i\leq l'\leq j}
\end{array}} [A^s_{k,l}-A^s_{k,l'}]_++
\sum_{\begin{array}{c}
 \vspace{-0.15cm}
 \scriptscriptstyle{1\leq k<k'<l\leq n}\\
 \vspace{-0.15cm}
 \scriptscriptstyle{(k,l)\in \UR(i,j)}\\
 \scriptscriptstyle{i\leq k'\leq j}
\end{array}} [A^s_{k,l}-A^s_{k',l}]_+
\ \leq\ 
n^3(\Gamma(A^s)-\Gamma(A^{s+1})).
\end{equation}
Repeating the above argument for elements of $\LL(i,j)$, the following inequality can be derived in a  similar fashion. 
\begin{equation}\label{eq:half2}
\sum_{\begin{array}{c}
 \vspace{-0.15cm}
 \scriptscriptstyle{1\leq k<k'<l'\leq n}\\
 \vspace{-0.15cm}
 \scriptscriptstyle{(k',l')\in \LL(i,j)}\\
 \scriptscriptstyle{1\leq k\leq i}
\end{array}} [A^s_{k,l'}-A^s_{k',l'}]_++
\sum_{\begin{array}{c}
 \vspace{-0.15cm}
 \scriptscriptstyle{1\leq k'<l'<l\leq n}\\
 \vspace{-0.15cm}
 \scriptscriptstyle{(k',l')\in \LL(i,j)}\\
 \scriptscriptstyle{j\leq l\leq n}
\end{array}} [A^s_{k',l}-A^s_{k',l'}]_+
\ \leq\ 
n^3(\Gamma(A^s)-\Gamma(A^{s+1})).
\end{equation}
Combining (\ref{equality-1}), (\ref{eq:half1}), and (\ref{eq:half2}), we conclude the following.
\begin{eqnarray*}
1_{\UR}(i,j) 0_{\LL}(i,j)&=&\sum_{\begin{array}{c}
 \vspace{-0.15cm}
 \scriptscriptstyle{(k,l)\in \UR(i,j)}\\
 \vspace{-0.15cm}
 \scriptscriptstyle{(k',l')\in \LL(i,j)}\\
\end{array}}
[A^s_{k,\ell}-A^s_{k',\ell'}]_+ \\
&=& \frac{1}{2}\sum_{\begin{array}{c}
 \vspace{-0.15cm}
 \scriptscriptstyle{(k,l)\in \UR(i,j)}\\
 \vspace{-0.15cm}
 \scriptscriptstyle{(k',l')\in \LL(i,j)}\\
\end{array}}([A^s_{k,\ell}-A^s_{k',\ell}]_+ +[A^s_{k,\ell}-A^s_{k,\ell'}]_+)\nonumber\\
&+& \frac{1}{2}\sum_{\begin{array}{c}
 \vspace{-0.15cm}
 \scriptscriptstyle{(k,l)\in \UR(i,j)}\\
 \vspace{-0.15cm}
 \scriptscriptstyle{(k',l')\in \LL(i,j)}\\
\end{array}}([A^s_{k',\ell}-A^s_{k',\ell'}]_+
 +[A^s_{k,\ell'}-A^s_{k',\ell'}]_+)\\
 &\leq& \frac{n}{2}\sum_{\begin{array}{c}
 \vspace{-0.15cm}
 \scriptscriptstyle{1\leq k<l'<l\leq n}\\
 \vspace{-0.15cm}
 \scriptscriptstyle{(k,l)\in \UR(i,j)}\\
 \scriptscriptstyle{i\leq l'\leq j}
\end{array}} [A^s_{k,l}-A^s_{k,l'}]_++
\frac{n}{2}\sum_{\begin{array}{c}
 \vspace{-0.15cm}
 \scriptscriptstyle{1\leq k<k'<l\leq n}\\
 \vspace{-0.15cm}
 \scriptscriptstyle{(k,l)\in \UR(i,j)}\\
 \scriptscriptstyle{i\leq k'\leq j}
\end{array}} [A^s_{k,l}-A^s_{k',l}]_+\\
 &+& \frac{n}{2}\sum_{\begin{array}{c}
 \vspace{-0.15cm}
 \scriptscriptstyle{1\leq k<k'<l'\leq n}\\
 \vspace{-0.15cm}
 \scriptscriptstyle{(k',l')\in \LL(i,j)}\\
 \scriptscriptstyle{1\leq k\leq i}
\end{array}} [A^s_{k,l'}-A^s_{k',l'}]_++
\frac{n}{2}
\sum_{\begin{array}{c}
 \vspace{-0.15cm}
 \scriptscriptstyle{1\leq k'<l'<l\leq n}\\
 \vspace{-0.15cm}
 \scriptscriptstyle{(k',l')\in \LL(i,j)}\\
 \scriptscriptstyle{j\leq l\leq n}
\end{array}} [A^s_{k',l}-A^s_{k',l'}]_+\\
&\leq& n^4(\Gamma_1(A^s)-\Gamma_1(A^{s+1})).
\end{eqnarray*}
Note that the factor $\frac{n}{2}$ in the above inequalities appear, because every fixed row or column of $A^s$ can have at most $n$ cells from $\LL(i,j)$ or $\UR(i,j)$.
This completes the proof of the claim. 
\end{proof}
}
Finally, it is easy to observe that $\|A^s-{A^{s+1}}\|_1\leq \frac{2(0_{\LL}(i,j)+1_{UR}(i,j))}{n^2}$. Without loss of generality, assume that $1_\UR\geq 0_\LL$. Then
$$\gm(A^s)-\gm(A^{s+1})\geq  \frac{0_{LL}(i,j)1_{UR}(i,j)}{n^4}=\left( \frac{0_\LL (i,j)}{4n^2}\right)\left( 4\frac{1_{UR}(i,j)}{n^2}\right)\geq \frac{t}{4n^2}{\|A^s-A^{s+1}\|_1}.$$
 Applying the above result, we get
\begin{eqnarray*}
\gm(A)-\gm(\widehat{A})&=&\sum_{s=0}^m (\gm(A^s)-\gm(A^{s+1}))\geq \sum_{s=0}^m\frac{t}{4n^2}{\|A^s-{A^{s+1}}\|_1}\\
&\geq&\frac{t}{4n^2}\|\sum_{s=0}^m A^s-{A^{s+1}}\|_1=\frac{t}{4n^2}\|A-\widehat{A}\|_1,
\end{eqnarray*}
which finishes the proof.
\end{proof}

We now have all the ingredients to prove Theorem \ref{thm:main}, by combining the above lemma with Theorem \ref{Alg:01}. For this, parameter $t$ must be tuned so that the bounds from the preprocessing step and from Algorithm \ref{Alg:pre} give the best possible result; it appears that the best choice is  $t=\blue{4^{-2/3}}\gm(A)^{2/3}n^2$. (Note that this value is smaller than the value used in \blue{Corollary \ref{cor:algo1withbound}}.) 
For this choice of $t$,  the distance between the input matrix $A$ and the Robinson similarity matrix returned by Algorithm \ref{Alg:01} when applied to the updated version of matrix $A$ (after being processed  by Algorithm \ref{Alg:pre}) is bounded by $\blue{26} \gm (A)^{1/3}$. As the exponent on $\gm (A)$ has been decreased from $1/4$ to $1/3$, the preprocessing step leads to a substantially better Robinson  approximation.

\begin{proof}[Proof of Theorem \ref{thm:main}]
First assume that $A\in \cA_n$ is a binary matrix with $\gm (A)=\epsilon \blue{>0}$. Let $\widehat{A}$ be the output of Algorithm \ref{Alg:pre}, with threshold $t=\blue{4^{-2/3}}\epsilon^{2/3}n^2$. From Lemma \ref{lemma:prep}, 
$${\|A-\widehat{A}\|_1}\leq \blue{4^{5/3}} \epsilon^{-2/3}(\gm(A)-\gm(\widehat{A}))\leq 4^{5/3} \epsilon^{1/3}.$$ 
From Lemma \ref{lem:preproc}, we have that $\widehat{A}$ satisfies Condition (\ref{prop:URxLL}) of Theorem \ref{thm:algo1} for our choice of $t$.
Let $R$ be the output of  Algorithm \ref{Alg:01} applied to $\widehat{A}$ with parameter $t=\blue{4^{-2/3}}\epsilon^{2/3}n^2$. Then, by Theorem \ref{thm:algo1}, we have
$${\|\widehat{A}-R\|_1}\leq \frac{16\sqrt{t}+4}{n}=\blue{4^{5/3}}\epsilon^{1/3}+\frac{4}{n}.$$
Combining these inequalities, we get $\|{A}-R\|_1\leq 2\cdot 4^{5/3}\epsilon^{1/3}+\frac{4}{n}\leq \blue{26} \epsilon^{1/3}$, where we used the fact that $\gm (A)=\epsilon \geq \frac{1}{n^3}$ in the last inequality. 

Next, assume that $A\in \cA_n$ is a general matrix, and let $A=\sum_{k =1}^m (s_k-s_{k-1})A^{(k )}$ be the decomposition of $A$ into layers of binary matrices as described in Equation (\ref{eq:decomp}). Let $\epsilon_k:=\gm(A^{(k)})$, and recall that $\epsilon=\sum_{k=1}^m (s_k-s_{k-1})\epsilon_k.$ For every $1\leq k\leq m$, we apply the process described in the above paragraph to  $A^{(k )}$ with parameter $t_k=\blue{4^{-2/3}}\epsilon_k^{2/3}n^2$, and we obtain  Robinson matrices $R^{(k)}$ such that
$$\|A^{(k)}-R^{(k)}\|_1\leq\blue{26} \epsilon_k^{1/3}.$$ 
Letting $R=\sum_{k =1}^m (s_k-s_{k-1})R^{(k)}$, we have
\begin{eqnarray*}
\|A-R\|_1&\leq &\sum_{k=1}^m(s_k-s_{k-1})\|A^{(k)}-R^{(k)}\|_1
\leq  \blue{26} \sum_{k=1}^m(s_k-s_{k-1})\epsilon_k^{1/3}
\leq  \blue{26} \epsilon^{1/3},
\end{eqnarray*}
where in the last inequality we used the fact that the function $f(x)=x^{1/3}$ is concave. 

Finally,
we observe that the outputs of Algorithm \ref{Alg:pre} and Algorithm \ref{Alg:01}, when applied to a binary matrix $A$, are again binary matrices. Therefore, the Robinson approximation $R$ of Theorem \ref{thm:main} is binary, when $A$ is a binary matrix.
\end{proof} 

\section{Conclusions and further work}
We defined a parameter $\gm$ which measures how much a matrix resembles a Robinson similarity matrix. We gave a polynomial time algorithm which takes as input a symmetric matrix $A$, and finds a Robinson matrix $R$ so that the normalized $\ell^1$-distance between $A$ and $R$ is bounded by $\blue{26} \gm (A)^{1/3}$. The motivation of our work is the application to the problem of seriation of noisy data. This problem can now be approached by solving instead the optimization problem: given a matrix $A$, find a permutation $\pi$ of the rows and columns of $A$ so that $\gm (A^{\pi})$ is minimized. 

Our construction method is based on a combinatorial algorithm that runs in polynomial time. However, this may not give the best possible such Robinson approximation. As remarked in the introduction, the problem of finding the best possible Robinson approximation, with error measured in $\ell_1$ norm, can be formulated as a linear program. An open problem is whether there exists a combinatorial algorithm for this task (as there exist for the $\ell_\infty $ norm.)

In future work, we propose to study this optimization problem, to attempt to find algorithms which solve or approximate the problem, and to determine their complexity. It is well-known that the second eigenvector of the Laplacian of the matrix, also known as the Fiedler vector, is effective in finding the correct permutation in seriation without error. We propose to study the relationship between the Fiedler vector and the parameter $\gm$ . 

An immediate next step is to test our algorithm on real data, and see whether, in practice, the algorithm outperforms the theoretical bound. As well, a clever implementation of Algorithm \ref{Alg:pre} will likely lead to improved efficiency.   

\section*{Acknowledgments}

The authors thank the anonymous referees, whose suggestions greatly improved the paper.

\renewcommand{\thepage}{}
\bibliographystyle{siamplain}

\end{document}